\documentclass[12pt]{amsart} % describes what sort of document to make
\usepackage{amssymb,amsmath} % bring in some nice math fonts
\usepackage{amsthm}
\usepackage[mathcal]{euscript} % this brings in another nice math font
\usepackage{fullpage} % take up the whole page
\usepackage[pdftex]{graphicx,color} % this would let me include pdf pictures
\usepackage{cite}

\providecommand{\norm}[1]{\left\lVert#1\right\rVert} % and for the norm
\providecommand{\dotp}[1]{\left\langle#1\right\rangle}

\providecommand{\floor}[1]{\left\lfloor#1\right\rfloor}

\providecommand{\smallmat}[1]{\left(\begin{smallmatrix}#1\end{smallmatrix}\right)}

\providecommand{\mathlarger}[1]{\mbox{\larger$#1$}}
\providecommand{\partialby}[2]{\frac{\partial #1}{\partial #2}}
\DeclareMathOperator{\spn}{span} 
\DeclareMathOperator{\Gal}{Gal}

\DeclareMathOperator{\Orth}{O}
\DeclareMathOperator{\Ad}{Ad}

\DeclareMathOperator{\tr}{tr}
\DeclareMathOperator{\so}{so}

\DeclareMathOperator{\charpoly}{charpoly}
\DeclareMathOperator{\proj}{proj}
\newcommand{\gal}{\ensuremath{\mathfrak{gal}}}
\newcommand{\g}{\ensuremath{\mathfrak{g}}}
\newcommand{\h}{\ensuremath{\mathfrak{h}}}
\newcommand{\R}{\ensuremath{\mathbb{R}}}

\newcommand{\C}{\ensuremath{\mathbb{C}}}

\newcommand{\U}{\ensuremath{\mathcal{U}}}
\newcommand{\z}{\ensuremath{\mathfrak{z}}}
\newcommand{\Exterior}{\bigwedge\nolimits}

\newtheorem{theorem}{Theorem}

\DeclareMathSizes{12}{12}{10}{8}
\author{Mathew Wolak}
\address{Tufts University}
\email{mathew.wolak@tufts.edu}
\thanks{The author wishes to thank Fulton Gonzalez for introducing him to the
subject, and Tomoyuki Kakehi for his support. This material is based on work
supported by the National Science Foundation under Grant No. OISE-1310973.}
\title{Bi-invariant differential operators on the Galilean group}
\begin{document}
\begin{abstract}
The Galilean group is the group of symmetries of Newtonian mechanics, with Lie algebra $\gal(n)$. We find algebraically independent generators for the center of the universal
enveloping algebra of $\gal(n)$ using coadjoint orbits.
\end{abstract}
\maketitle
\section{Introduction}
The Galilean group $\Gal(n)$ is the Lie group of transformations between
reference frames in Newtonian mechanics in $n$ dimensional space. It is
generated by spatial rotations, translations in space and time, and
boosts, which correspond to changes in velocity (see \cite{mmcm}).
Elements of $\Gal(n)$
can be represented as $(n+2)\times (n+2)$ matrices 
\begin{align}
\label{GalGroup}
\left(\begin{array}{c|c}
\mathlarger{\mathlarger{\rho}} & \begin{smallmatrix}v_1 & x_1 \\\vdots &
\vdots\\v_n & x_n\end{smallmatrix} \\ 
\hline 
0 & \begin{smallmatrix} 1 &x_0\\0 & 1\end{smallmatrix}
\end{array}\right)
\end{align}

where $\rho$ is an element of $\Orth(n)$,
$\left(\begin{smallmatrix}x_1\\\vdots \\
x_n\end{smallmatrix}\right)$ the spatial translation, $x_0$ the time
shift, and $\left(\begin{smallmatrix}v_1\\ \vdots
\\ v_n\end{smallmatrix}\right)$
is the boost.

To get the action on $n+1$-dimensional spacetime, we first identify
affine $(n+1)$-space with the plane $x_{n+2} = 1$ in $\R^{n+2}$.
Then the action on a point $(X_1,\ldots,X_n,T) \in \R^{n+1}$ is:
\[
\left(\begin{array}{c|c}
\mathlarger{\mathlarger{\rho}} & \begin{smallmatrix}v_1 & x_1 \\\vdots &
\vdots\\v_n & x_n \end{smallmatrix}\\
\hline 
0 & \begin{smallmatrix} 1 &x_0\\0 & 1\end{smallmatrix}
\end{array}\right)\left(\begin{smallmatrix}
X_1 \\
\vdots \\
X_n \\
T \\
1
\end{smallmatrix}\right)
 = \left( \begin{smallmatrix}
X_1' \\
\vdots \\
X_n' \\
T' \\
1
\end{smallmatrix}\right)
\]

Taking derivatives, the Lie algebra $\gal(n)$ of $\Gal(n)$ is identified with
the set of matrices of the form
\begin{align}
\label{GalAlgebra}
\left(\begin{array}{c|c}
\mathlarger{\mathlarger{K}} & \begin{smallmatrix}v_1 & x_1 \\\vdots & \vdots\\v_n & x_n\end{smallmatrix}\\
\hline 
0 & \begin{smallmatrix} 0 &x_0\\0 & 0\end{smallmatrix}
\end{array}\right)
\end{align}
where $K$ is an element of $\so(n)$.

The Galilean group is an important example of a {\em contraction}. 
Given a Lie algebra $\g = (V,[\cdot,\cdot])$, and a subalgebra $\h =
(U,[\cdot,\cdot])$, choose a complementary subspace $W$ so that
$V = U \oplus W$.
For $\epsilon > 0$, let $T_\epsilon$ be the linear operator\[
T_\epsilon (u + w) = u + \epsilon w \; \; \; \;(u \in U, w \in W)
\]
One can define a new bracket $[\cdot,\cdot]_\epsilon$ on $V$ by
\[
[X,Y]_\epsilon = T_\epsilon^{-1}[T_\epsilon X,T_\epsilon Y]
\] The Lie algebra
$(V,[\cdot,\cdot]_\epsilon)$ is isomorphic to $\g$. As long as $\h$ is a
subalgebra, we may define $[\cdot,\cdot]_0 = \lim\limits_{\epsilon
\rightarrow 0} [\cdot,\cdot]_\epsilon$. The resulting Lie algebra
$\g_0 = (V,[\cdot,\cdot]_0)$ is in general not isomorphic to $\g$, and is
known as an In\"on\"u-Wigner contraction \cite{gilmore}. Note that the
subalgebra $\h_0 = (U, [\cdot,\cdot]_0)$ is isomorphic to $\h$, and that
$(W,[\cdot,\cdot]_0)$ is abelian. The contraction process can be seen
as `flattening' $W$.

The Galilean group is a contraction of the
Poincar\'e group, and this reflects the fact that special relativity
becomes Newtonian mechanics in the limit as the speed of light goes to
infinity. The Poincar\'e group is itself a contraction, whose bi-invariant
differential operators were found in \cite{takiff} and \cite{fgonz88}.

The universal enveloping algebra $\U(\g)$ of a Lie algebra can be
identified with the
algebra of left-invariant differential operators on the corresponding Lie
group $\mathbf{D}(G)$. If $S(\g)$ is the symmetric algebra on $\g$, there
is a linear bijection $\lambda : S(\g) \rightarrow \mathbf{D}(G)$ defined by
\[
(\lambda(P)f)(g) = P(\partial_1,\ldots,\partial_n)f(g \exp(t_1X_1 +
\dots + t_nX_n))
\]

where $\partial_i = \partialby{}{t_i}$, and $X_1,\ldots,X_n$ is a
basis of $\g$. \cite[ch II, Theorem 4.3]{helgGGA}

This bijection commutes with the adjoint action of $G$ on $S(\g)$ and
$\mathbf{D}(G)$, and so we have, for $P \in S(\g)^{G}$: \[
\Ad(g)\lambda(P) = \lambda(\Ad(g)P) = \lambda(P)
\]
which shows that the image of an $\Ad(G)$ invariant polynomial is itself
$\Ad(G)$ invariant. Since the $\Ad(G)$ invariant elements of $\U(\g)$ are
exactly  $\z(\U(\g))$, we have identified $S(\g)^{G}$ with $\z(\U(\g))$.

Unfortunately, while $\lambda$ is a linear bijection, it isn't an isomorphism
of algebras, and so in general, $\lambda(PQ)  \neq \lambda(P)\lambda(Q)$.
However we do have that: \begin{align}
\label{highdeg}
\deg(\lambda(PQ) - \lambda(P)\lambda(Q)) < \deg(\lambda(PQ))
\end{align}
This allows us to show inductively that if $\{P_1, \dots,P_m\}$ generate
$S(\g)^G$, then $\{\lambda(P_1),\ldots,\lambda(P_m)\}$ generate $\z(\U(\g))$.
The degree zero case is trivial.
If $D \in \z(\U(\g))$, then we can write \[
D = \lambda(q(P_1,\ldots,P_m))
\]
for some polynomial $q$. Then \[
D - q(\lambda(P_1),\ldots,\lambda(P_m))
\]
is a central element whose degree is less than $\deg(D)$ by \eqref{highdeg}.
By induction, $\z(\U(\g))$ is generated by
$\{\lambda(P_1,\ldots,\lambda(P_m)\}$.

Elements of $S(\g)$ can be viewed as polynomials on $\g^*$, the dual
space of $\g$. With this identification, $\Ad(g)P(X^*) =
P(\Ad^*(g^{-1})X^*)$,
where $\Ad^*$ is the coadjoint representation of $G$ on $\g^*$.
A polynomial which is invariant under the adjoint representation will be
constant on coadjoint orbits. This allows us to use the tools of classical
invariant theory to solve the problem.

\section{Restriction of the Problem}

We would like to find a subspace $S$ of $\g^*$ for which the set
$\{Ad^*(g)s | g \in G, s \in S\}$ is Zariski dense in $\g^*$. In such a case,
any polynomial which is invariant under the coadjoint action is defined
by its values on that subspace. In particular, we would like to find
an $S$ which is {\em transversal} to the coadjoint orbits.
Since in general orbits
will intersect the transversal subspace multiple times, the restriction of
an invariant function will satisfy some further restrictions, namely that it be
invariant under the the action of the subgroup of $G$ which fixes the
subspace.

Any polynomial function on $\g^*$ will restrict to a polynomial function
on a subspace, but the converse is not true. Therefore, after identifying
candidate polynomials, we must eliminate those which don't extend to a
polynomial function.

For a semisimple Lie group, the Killing form provides a nondegenerate inner
product with which we can identify the Lie algebra and its dual. Because the
Killing form is invariant with respect to the adjoint action, the coadjoint
action is essentially the same as the adjoint action. Unfortunately, $\Gal(n)$
does not have such a convenient inner product. Instead, we use the usual matrix
inner product $\dotp{A,B} = \tr(A^TB)$.

\section{The Coadjoint Action}
Recalling the matrix forms of the Galilean group and its Lie algebra given in
\eqref{GalGroup} and \eqref{GalAlgebra}, we must now describe the dual space
$\gal^*$ and the coadjoint action of $\Gal(n)$ on $\gal^*$.

To represent $\gal(n)^*$, we use the usual matrix inner product
$A^*(B) = \tr(A^T B)$. We can then identify $\gal(n)^*$ with
$\R^{(n+2)\times (n+2)}$ modulo $\gal(n)^\perp$, where
\[
\gal(n)^\perp = \left\{
\left(\begin{array}{c|c}
\mathlarger{\mathlarger{S}} & \begin{smallmatrix}0 & 0\\\vdots & \vdots\\
0 & 0\end{smallmatrix} \\ \hline
A & \begin{smallmatrix}a & 0 \\b & c\end{smallmatrix}
\end{array}\right) \mid
S \text{ symmetric}, A \in \R^{2\times n}, a,b,c \in \R
\right\}
\]
Using the fact that the trace is invariant under cyclic permutations,
\begin{align*}
(\Ad^*(g)A^*)(B) &= A^*(\Ad(g^{-1}B) \\
&= A^*(g^{-1}Bg) \\
&= \tr(A^Tg^{-1}Bg) \\
&= \tr(gA^Tg^{-1}B) \\
&= ((gA^Tg^{-1})^T)^*(B)
\end{align*}

Every element of $\Gal(n)$ can be written as $\tau\rho$, where
$\rho$ is a spatial rotation and $\tau$ is a space-time translation and boost.
Because of this, we may
consider the actions of rotations separately from boosts and translations.

If \[
\rho = \left(\begin{array}{c|c}
\rho & \begin{smallmatrix}0 & 0\end{smallmatrix} \\
\hline
0 & \begin{smallmatrix}1 & 0 \\ 0 & 1\end{smallmatrix}
\end{array}\right)
\]

and \begin{align}
\label{genelt}
A^* = \left(\begin{array}{c|c}
\mathlarger{\mathlarger{K^*}} & \begin{smallmatrix}v^*_1 & x^*_1 \\\vdots & \vdots\\v^*_n & x^*_n\end{smallmatrix}\\
\hline 
0 & \begin{smallmatrix} 0 &x^*_0\\0 & 0\end{smallmatrix}
\end{array}\right)
\end{align}

then a straightforward calculation shows:
\[
\rho^{-1} = \left(\begin{array}{c|c}
\rho^{-1} & \begin{smallmatrix}0 & 0\end{smallmatrix} \\
\hline
0 & \begin{smallmatrix}1 & 0 \\ 0 & 1\end{smallmatrix}
\end{array}\right)
\]
and
\begin{align}
\label{rotate}
(\rho A^{*T}\rho^{-1})^T = \left(\begin{array}{c|c}
\rho K^*\rho^{-1} & \rho\smallmat{v_1^* & x_1^*\\ \vdots & \vdots \\
v_n^* & x_n^*} \\ \hline
0 & \begin{smallmatrix} 0 & x_0^* \\ 0 & 0\end{smallmatrix}
\end{array}\right)
\end{align}

For translations and boosts, \begin{align}
\label{transboost}
\begin{split}
\tau &= \left(\begin{array}{c|c}
I & \begin{smallmatrix}v_1 & x_1 \\ \vdots & \vdots \\ v_n & x_n\end{smallmatrix}
\\ \hline
0 & \begin{smallmatrix}1 & x_0 \\ 0  & 1\end{smallmatrix}
\end{array}\right) \\
\tau^{-1} &= \left(\begin{array}{c|c}
I & \begin{smallmatrix}-v_1 & v_1x_0 - x_1 \\ \vdots & \vdots \\
-v_n & v_nx_0 - x_n\end{smallmatrix} \\ \hline
0 & \begin{smallmatrix}1 & -x_0 \\ 0 & 1\end{smallmatrix}
\end{array}\right) \\
(\tau A^*\tau^{-1})^T &= \left(\begin{array}{c|c}
K^* + \smallmat{v_1^* & x_1^* \\ \vdots & \vdots \\v_n^* & x_n^*}
\smallmat{v_1 \dots v_n\\x_1 \dots x_n} &
\begin{smallmatrix}v_1^* + x_0x_1^* & x_1^* \\ \vdots & \vdots \\v_n^* + x_0x_n^* & x_n^*\end{smallmatrix} \\ \hline
* & \begin{smallmatrix}* & x_0^* + \sum v_ix_i^* \\ * & *\end{smallmatrix}
\end{array}\right)\\
&\equiv \left(\begin{array}{c|c}
K^* + \frac{1}{2}\smallmat{v_1^* & x_1^* \\ \vdots & \vdots \\v_n^* & x_n^*}
\smallmat{v_1 \dots v_n\\x_1 \dots x_n} -
\frac{1}{2} \smallmat{v_1 \dots v_n\\x_1 \dots x_n}^T
\smallmat{v_1^* & x_1^* \\ \vdots & \vdots \\v_n^* & x_n^*}^T &
\begin{smallmatrix}v_1^* + x_0x_1^* & x_1^* \\ \vdots & \vdots \\v_n^* + x_0x_n^* & x_n^*\end{smallmatrix} \\ \hline
0 & \begin{smallmatrix}0 & x_0^* + \sum v_ix_i^* \\ 0 & 0\end{smallmatrix}
\end{array}\right)
\end{split}
\end{align}

where the final congruence is modulo $\gal(n)^\perp$

To find a transversal subspace of $\gal(n)^*$, let $A^*$ be a generic
element written as as \eqref{genelt}. We would like to find a $g$ which can take
$(g^{-1}A^{*T}g)^T$ to our subspace. 

First, use a rotation \eqref{rotate} to put $x^*$ in the form
$(A 0 \dots 0)^T$ and $v^*$ 
in the form $(C B 0 \dots 0)^T$. We can then use an $x_0$-only translation
element (equation \eqref{transboost} with $x_0 = -\frac{v_1^*}{x_1^*}$) to zero out the first element of $v^*$. Thus we can restrict our
attention to matrices of the form \[
\left(\begin{array}{c|c}
K^* & \begin{smallmatrix}0 & A\\ B & 0\\ 0 & 0 \\\vdots & \vdots\end{smallmatrix}
\\ \hline
0 & \begin{smallmatrix} 0 & x_0^* \\ 0 & 0\end{smallmatrix}
\end{array}\right)
\]

Let's now turn to the top-left quadrant. Referring again to \eqref{transboost},
\begin{align*}
\frac{1}{2}\smallmat{0 & A \\ B & 0 \\ 0 & 0 \\\vdots & \vdots}
\smallmat{v_1 &  \dots & v_n \\ x_1 & \dots & x_n} -
\frac{1}{2}\left(\smallmat{0 & A \\ B & 0 \\ 0 & 0 \\\vdots & \vdots}
\smallmat{v_1 &  \dots & v_n \\ x_1 & \dots & x_n}\right)^T  \\ 
= \frac{1}{2}\smallmat{0 & Ax_2 - Bv_1 & Ax_3 & Ax_4 & \dots & Ax_n\\
Bv_1 - Ax_2 & 0 & Bv_3 & Bv_4 &\dots & Bv_n \\
-Ax_3 & -Bv_3 & 0 & 0 & \dots & 0\\
-Ax_4 & -Bv_4 & 0 & 0 & \dots & 0\\
\vdots &\vdots &\vdots &\vdots &\ddots & \vdots\\
-Ax_n & -Bv_n & 0 & 0 & \dots & 0}
\end{align*}
Appropriate choices for the $x_i$'s and $v_i$'s allow us to zero out the
topmost two rows and leftmost two columns of $K^*$. In particular, for $i \geq
3$, we set $x_i = -\frac{2K^*_{1i}}{A}$ and $v_i -\frac{2K^*_{2i}}{B}$. It is
important to note that there is a remaining degree of freedom when zeroing
$K^*_{12}$, as this will allow us to clear $x^*_0$ by setting $v_1 = 
-\frac{x_0^*}{x_1^*}$ in \eqref{transboost}.

We are free to choose a rotation from the subgroup of $\Orth(n)$ which fixes
the first two coordinates. We can use this freedom to conjugate the upper-left
block to an element of a maximal torus in $\so(n-2)$.

Our transversal subspace is made of matrices of the form\[
\left(\begin{array}{c|c}
\begin{smallmatrix} 0 & 0 & \dots \\ 0 & 0 & \dots \\ \vdots & \vdots & K
\end{smallmatrix} & 
\begin{smallmatrix} 0 & A \\ B & 0 \\ 0 & 0 \\ \vdots & \vdots
\end{smallmatrix} \\ \hline
0 & \begin{smallmatrix}0 & 0 \\ 0 & 0\end{smallmatrix}
\end{array}\right)
\]
where $A$ and $B$ are real numbers, and $K$ is an element of a $\so(n-2)$
maximal torus.

\section{Case $n = 1$}

When $n = 1$, the upper-left hand block is always zero.
We use $x_0$ to zero out the $v_1^*$ and $v_1$ to zero out $x_0^*$. Thus,
a generic element \[
\left(\begin{array}{c c c} 0 &v_1^* & x_1^* \\0 &0& x_0^*\\ 0 & 0 & 0
\end{array}\right)
\]
of $\gal(1)^*$ can be conjugated to\[
\left(\begin{array}{c c c} 0 &0 & x_1^* \\0 &0& 0\\ 0 & 0 & 0
\end{array}\right)
\]

Since $\left(-1\right) \in \Orth(1)$, the algebra of $\Ad^*$-invariant
polynomials is generated by $X_1^2$
\section{Case $n = 2,3$}

\begin{theorem}
For $n \in \{2,3\}$, $S(\gal(n))^{\Gal(n)}$ is generated by $\sum\limits_i
X_i^2$ and \linebreak
$\left(\sum\limits_i X_i^2\right)\left(\sum\limits_i V_i^2\right) -
\left(\sum\limits_i X_iV_i\right)^2$
\end{theorem}
\begin{proof}
When $n = 2$ or $3$, $K^*$ may not be zero, but
we can still zero out the upper-left block entirely. If \[\Xi =
\left(
\begin{array}{c|cc}
K^* &v^* & x^* \\  \hline
0 & 0 & t^*  \\ 0 & 0 & 0\\
\end{array}\right)
\] is a matrix representing an element of $\gal(n)^*$, then
Then there exists an element $g \in \Gal(n)$ such that
\[
\Ad(g)^*\Xi \equiv 
\left(\begin{array}{c|cc}
0 & \begin{smallmatrix} 0 \\ B \end{smallmatrix} & \begin{smallmatrix} A \\ 0
\end{smallmatrix} \\ \hline 0 & 0 & 0\\
0 & 0 & 0 \\
\end{array}\right)
\]
where $A = \norm{x^*}$ and $B = \norm{\proj_{x^{*\perp}}(v^*)}$. Conjugating
by an element of $\Orth(n)$ can switch the signs of $A$ and $B$ independently,
so the
invariant polynomials on the transverse manifold are generated by $A^2$ and
 $B^2$. Thus any polynomial $P$ which is invariant under the coadjoint
action of $\Gal(n)$ restricts to a polynomial $\overline{P} = F(A^2,B^2)$
defined on the transverse manifold.

Define polynomials $Q_1 : \gal(n)^* \rightarrow
\R$ and $Q_2 : \gal(n)^* \rightarrow \R$, which take an element of $\gal(n)^*$
to the values of $A^2$ and $A^2B^2$ on the intersection of its $\Ad^*$ orbit
with the transversal submanifold. If $\Xi$ is a matrix representing
an element of $\gal(n)^*$, we have:
\begin{align*}
Q_1(\Xi) &= \norm{x^*}^2 = \left(\sum\limits_iX_i^{2}\right)(\Xi) \\
Q_2(\Xi) &= \norm{x^*}^2\norm{\proj_{x^{*\perp}}v^*}^2 = \left(
\left(\sum\limits_i X_i^{2}\right)\left(\sum\limits_i V_i^{2}\right)
- \left(\sum\limits_i X_i V_i\right)^2\right)(\Xi)
\end{align*}
Because the restriction is injective on the space of $\Ad^*$ invariant
polynomials, we know that \begin{align}
P(\Xi) &= F\left(Q_1(\Xi),\frac{Q_2(\Xi)}{Q_1(\Xi)}\right) \\
\label{23Fbreakdown}
&= F_1(Q_1(\Xi),Q_2(\Xi)) + \frac{F_2(Q_1(\Xi),Q_2(\Xi))}{Q_1(\Xi)^\ell}
\end{align}
where $F_2$ is indivisible by its first argument.

While we are considering only real Lie algebras, if $P$ is a polynomial,
it must extend to a polynomial on the complexification of $\gal(n)^*$. With
this in mind, consider \[
\Xi = \left(\begin{array}{cc|cc}
0 & 0 & 0 & i\\
0 & 0 & z & 1 \\
\hline
0 & 0 & 0 & 0\\
0 & 0 & 0 & 0
\end{array}\right)
\]
for $n = 2$ or
\[
\Xi = \left(\begin{array}{ccc|cc}
0 &0 & 0 & 0 & i\\
0 &0 & 0 & z & 1 \\
0 &0 & 0 & 0 & 0 \\
\hline
0 &0 & 0 & 0 & 0\\
0 &0 & 0 & 0 & 0
\end{array}\right)
\]
when $n = 3$. In either case, $Q_1(\Xi) = 0$, and $Q_2(\Xi) = z^2$.
\eqref{23Fbreakdown} then becomes \begin{align*}
P(\Xi) = F_1(0,z^2) + \frac{F_2(0,z^2)}{0^\ell}
\end{align*}
Since $P$ was assumed to be a polynomial, either $\ell = 0$, or $F_2(0,z^2) =
0$ for all $z \in \C$. Since $F_2$ was assumed to be indivisible by its first
element, this implies that either $\ell = 0$ or $F_2$ is the constant function
$0$. Therefore, $P$ is a polynomial in $Q_1$ and $Q_2$
\end{proof}
\section{Case $n > 3$}
\begin{theorem}
When $n > 3$, 
$S(\gal(n))^{\Gal(n)}$ is generated
by 
the sums of determinants of $2k\times 2k$ submatrices formed by taking the
symmetric minors of \[
\left(\begin{array}{c|c}
\mathlarger{\mathlarger{K^*}} & \begin{smallmatrix}v^*_1 & x^*_1 \\\vdots & \vdots\\v^*_n & x^*_n\end{smallmatrix}\\
\hline 
\begin{smallmatrix}-v^{*T}\\-x^{*T}\end{smallmatrix} &
\begin{smallmatrix} 0 &0\\0 & 0\end{smallmatrix}
\end{array}\right)
\] which include the last two rows and columns.
\end{theorem}
\begin{proof}
When $n >3$, the coadjoint action can no longer zero out the entire
upper-left block, only the uppermost two rows and leftmost two columns. The
transversal subspace $S$ is made of matrices of the form\[
\left(\begin{array}{c|c}
\begin{smallmatrix} 0 & 0 & \dots \\ 0 & 0 & \dots \\ \vdots & \vdots & K
\end{smallmatrix} & 
\begin{smallmatrix} 0 & A \\ B & 0 \\ 0 & 0 \\ \vdots & \vdots
\end{smallmatrix} \\ \hline
0 & \begin{smallmatrix}0 & 0 \\ 0 & 0\end{smallmatrix}
\end{array}\right)
\]
where $K$ is an $\so(n-2)$ maximal torus.
Conjugating by elements of $\Orth(n)$ can change the signs of $A$ and
$B$ independently,
as well as permuting and changing signs of the elements of $K$.
The polynomials
on $S$ invariant under the subgroup of $\Gal(n)$ which fixes $S$ are generated
by $A^2$, $B^2$, and the coefficients of the characteristic polynomial of $K$
\cite{KNII}.

By
the same argument as in the $n = 2$ case, $B^2$ must be multiplied by $A^2$
to clear the denominator. This gives invariant polynomials $A^2 = \sum X_i^2$
and $A^2B^2 = (\sum X_i^2)(\sum V_i^2) - (\sum X_iV_i)^2$

We now turn to the polynomials which depend on the $K$ part. When conjugating
to $S$, one step was to zero out the rows and columns corresponding to
$x^*$ and $v^*$. This is equivalent to pre- and post- multiplying
by the matrix $P$,
the orthogonal projection to $\spn\{x^*,v^*\}^\perp$.
The resulting element of $\so(n-2)^*$ is subject to a $\Orth(n-2)$ action, and
it is well-known that the resulting invariant polynomials generated by
the coefficients of the characteristic polynomial. Thus we are looking for
$\charpoly(PK^*P)$

We will also make use of some facts about exterior algebras. Suppose
that $e_1,..,e_m$ and $f_1,\ldots,f_n$ are bases for vector spaces
$V_1$ and $V_2$ respectively. 
If $A : V_1 \rightarrow V_2$ is a linear function, then define
$\Exterior^kA : \Exterior^kV_1 \rightarrow \Exterior^kV_2$ to be the map
$x_1\wedge x_2\wedge\cdots\wedge x_k \mapsto Ax_1\wedge Ax_2 \wedge \cdots
\wedge Ax_k$. Then the $(e_{i_1}\wedge\cdots\wedge e_{i_k},f_{j_1}\wedge\cdots
\wedge f_{j_k})$ element of the matrix of $\Exterior^k A$ is
the determinant of the $k\times k$ matrix formed by taking elements in rows
$i_1,\ldots,i_k$ and columns $j_1,\ldots,j_k$. In particular, the coefficient
of the $x^{n-k}$ term of $\charpoly(A)$ is $\tr(\Exterior^k A)$.

Let \[\omega = \frac{v\wedge x}{\norm{v\wedge x}} = \frac{1}{AB}v\wedge x\]
(the sign of $AB$ is ambiguous, WLOG we may assume it's positive),
and let
\[
K' = \left(\begin{array}{c|c}
\mathlarger{\mathlarger{K^*}} & \begin{smallmatrix}v^*_1 & x^*_1 \\\vdots & \vdots\\v^*_n & x^*_n\end{smallmatrix}\\
\hline 
\begin{smallmatrix}-v^{*T}\\-x^{*T}\end{smallmatrix} &
\begin{smallmatrix} 0 &0\\0 & 0\end{smallmatrix}
\end{array}\right)
\]
Let $e_i$ be the vector (written vertically) with a $1$ in the $i$'th row
and zeros elsewhere.

If $y \in \spn\{e_1,\ldots,e_n\}^\perp$:\begin{align}
K'y = K^*y - (y\cdot v^*)e_{n+1} - (y\cdot x^*)e_{n+2}\end{align}
And we also have the following:
\begin{align}
\label{epart}
\left(\Exterior^2K'\right)e_{n+1}\wedge e_{n+2} = v^*\wedge x^* = AB\omega
\end{align} and
\begin{align}
\begin{split}
&\left(\Exterior^2K'\right)\omega \\ &= \frac{1}{AB}\left(\Exterior^2K'\right)v^*\wedge
x^* \\
&= \frac{1}{AB}(K^*v^* - (v^* \cdot v^*)e_{n+1} - (v^*\cdot x^*)e_{n+2})\wedge
(K^*x^* - (x^*\cdot v^*) e_{n+1} - (x^*\cdot x^*) e_{n+2}) \\
&= \frac{1}{AB}\left[K^*v^*\wedge(K^*x^* - (x^*\cdot v^*) e_{n+1} -
(x^*\cdot x^*) e_{n+2}) \right.\\
&+ (K^*v^* - (v^* \cdot v^*)e_{n+1} - (v^*\cdot x^*)e_{n+2})\wedge K^*x^* \\
&+\left((x^*\cdot x^*)(v^*\cdot v^*) - (x^*\cdot v^*)^2\right)e_{n+1}\wedge
e_{n+2}\left.\right]
\end{split}
\end{align}
In particular, note that the $e_{n+1}\wedge e_{n+2}$ term of
$(\Exterior^2K')\omega$ is $ABe_{n+1}\wedge e_{n+2}$.

Now consider any diagonal element $(\Exterior^{k+4}K')_{(I,I)}$ for which
$I = (i_1,\ldots,i_{k+2},n+1,n+2)$. By \eqref{epart},
$(\Exterior^{k+4}K')\alpha\wedge e_{n+1}\wedge e_{n+2} \in
\Exterior^{k+2}\R^{n+2}\wedge\omega$. Since we're only considering diagonal
elements, this means we can restrict our attention to the subspace
\[
\left(\Exterior^k\R^n\right)\wedge\omega\wedge e_{n+1}
\wedge e_{n+2}
\]
where by $\R^n$ refers to $\spn\{e_1,\ldots,e_n\} \subset \R^{n+2}$. Recall that
$P$ projects to $\spn\{x^*,v^*\}^\perp$. The we have, for $y_1,\ldots,y_k \in
\R^n$:
\begin{align}
\label{finalbit}
\begin{split}
&\left(\Exterior^{(k+4)}K'\right)y_1\wedge\cdots\wedge y_k\wedge\omega\wedge
e_{n+1}\wedge e_{n+2} \\
=& \left(\Exterior^{(k+4)}K'\right)Py_1\wedge\cdots\wedge Py_k\wedge\omega\wedge
e_{n+1}\wedge e_{n+2} \\
=& A^2B^2 K^*Py_1\wedge\cdots\wedge K^*Py_k\wedge\omega\wedge e_{n+1}\wedge
e_{n+2} \\
&\quad + \text{terms without $\omega\wedge e_{n+1}\wedge e_{n+2}$} \\
=& A^2B^2PK^*Py_1\wedge\cdots\wedge PK^*Py_k\wedge\omega\wedge e_{n+1}
\wedge e_{n+2} \\
&\quad + \text{terms without $\omega\wedge e_{n+1}\wedge e_{n+2}$}
\end{split}
\end{align}
By \eqref{finalbit}, \[
\sum\limits_{I = (i_1,\ldots,i_k+2,n+1,n+2)}
\left(\Exterior^{k+4}K'\right)_{(I,I)} =
A^2B^2\tr\left(\Exterior^{k}PK^*P\right)
\]
which implies that the sum of $(k+4)\times(k+4)$ subdeterminants
which include the last two
rows and columns is the $x^{n -k}$ coefficient of $A^2B^2\charpoly(PK^*P)$.
Because $PK^*P$ is skew-symmetric, the odd-$k$ terms will be zero.

To summarize, we have the polynomials $Q_1 = x^*\cdot x^*$, $Q_2 =
(v^*\cdot v^*)(x^*\cdot x^*) - (x^*\cdot v^*)^2$, and
$Q_3,\ldots,Q_{2+\floor{n/2}}$, which are $Q_2$ times the characteristic
polynomial coefficients of $PK^*P$

Finally, we must show that the $Q_i$ generate the invariant polynomials.
Suppose that $F(X)$ is an invariant polynomial not generated by the $Q_i$.
Restricting to the transversal subspace $S$, $\overline{F}$ is generated
by the restrictions of $Q_1, \frac{Q_2}{Q_1}, \frac{Q_3}{Q_2},\ldots,
\frac{Q_{2+\floor{n/2}}}{Q_2}$. By the injectivity of restriction to $S$,
this means that
\begin{align}
\label{ratbreakdown}\begin{split}
F(X) &= F'\left(Q_1(X), \frac{Q_2}{Q_1}(X), \frac{Q_3}{Q_2}(X),\ldots,
\frac{Q_{2+\floor{n/2}}}{Q_2}(X)\right)\\
&= F''(Q_1(X),Q_2(X),\ldots,Q_{2+\floor{n/2}}(X)) +
\frac{\tilde{F}(Q_1(X),Q_2(X),\ldots,Q_{2+\floor{n/2}}(X))}
{Q_1(X)^kQ_2(X)^\ell}
\end{split}
\end{align}
where $\tilde{F}$ is assumed not to be divisible by its first or second
arguments.
To show that $F$ is generated by the $Q_i$, we must show that $k = \ell = 0$.

While all of the preceeding work was done over $\R$, any polynomial extends
to a polynomial on $\C$. Consider matrices of the form\[
X_\Xi = \left(\begin{array}{cc|ccc|cc}
0&0&0 & \dots & 0 & 0 & 1\\
0&0&0& \dots & 0  & 1 & i\\
\hline
0&0&&& & 0 & 0\\
\vdots&\vdots& & \Xi & & \vdots & \vdots \\
0&0& & && 0 & 0\\ \hline
0&0&0 & \dots & 0 & 0 &0\\
0&0&0 & \dots & 0 & 0 &0
\end{array}\right)
\]
where $\Xi$ is an $(n-2)\times(n-2)$ skew-symmetric matrix.

For any $\Xi$, $Q_1(X_\Xi) = 0$ and $Q_2(X_\Xi) = 1$. For $3 \leq i \leq
2+\floor{n/2}$, we get the nonconstant
characteristic polynomial coefficients for
$\Xi$, which are known to be algebraically independent. Because of this
algebraic independence, and that $\tilde{F}$ was assumed not to be divisible
by its first argument, there exists a $\Xi$ such that the numerator of
\eqref{ratbreakdown} is nonzero, while $Q_1(X_\Xi)$ is zero. Therefore,
$k = 0$.

To show that $\ell = 0$, let \[
Y_\Xi = \left(\begin{array}{cc|ccc|cc}
0&0&0 & \dots & 0 & 0 & 1\\
0&0&0& \dots & 0  & 1 & 0\\
\hline
0&0&&& & i & 0\\
0 & 0 &&& & 0 & 0\\
\vdots&\vdots& & \Xi & & \vdots & \vdots \\
0&0& & && 0 & 0\\ \hline
0&0&0 & \dots & 0 & 0 &0\\
0&0&0 & \dots & 0 & 0 &0
\end{array}\right)
\]

For any $\Xi \in \so(n-2)$, $Q_1(Y_\Xi) = 1$ and $Q_2(Y_\Xi) = 0$. For $j
\geq 3$, $Q_j(Y_\Xi)$ is the sum of all $2j\times2j$
subdeterminants of \[
Y'_\Xi = \left(\begin{array}{cc|ccc|cc}
0&0&0 & \dots & 0 & 0 & 1\\
0&0&0& \dots & 0  & 1 & 0\\
\hline
0&0&&& & i & 0\\
0 & 0 &&& & 0 & 0\\
\vdots&\vdots& & \Xi & & \vdots & \vdots \\
0&0& & && 0 & 0\\ \hline
0&-1&-i & \dots & 0 & 0 &0\\
-1&0&0 & \dots & 0 & 0 &0
\end{array}\right)
\]
which include the leftmost two columns and bottom two rows.

Consider a skew-symmetric submatrix of $Y'_\Xi$ which includes the last
two rows and columns. If the submatrix doesn't also include the first row
and column, its leftmost column, and thus its determinant, will be zero.
Similarly, any nonzero such minor must also include at least one of the
second and third rows.

If the submatrix includes the second row and column, the corresponding
minor is 
\[
\det\left(\begin{array}{cc|ccc|cc}
0&0&0 & \dots & 0 & 0 & 1\\
0&0&0& \dots & 0  & 1 & 0\\
\hline
0&0&&& & \alpha & 0\\
0 & 0 &&& & 0 & 0\\
\vdots&\vdots& & \hat{\Xi} & & \vdots & \vdots \\
0&0& & && 0 & 0\\ \hline
0&-1&-\alpha & \dots & 0 & 0 &0\\
-1&0&0 & \dots & 0 & 0 &0
\end{array}\right) = \det(\hat{\Xi})
\]
where $\hat{\Xi}$ is a $(2j - 4)\times(2j-4)$ submatrix of $\Xi$, and
$\alpha$ is either $0$ or $i$ depending on whether the third column is
included.

If the submatrix does not include the second row and column (which can
only happen for $n > 4$), it must
include the third, and the minor
is instead \[
\det\left(\begin{array}{cc|ccc|cc}
0&0&0 & \dots & 0 & 0 & 1\\
0&0&\xi_{1k_1}& \dots & \xi_{1k_{2k-4}}  & i & 0\\
\hline
0&-\xi_{1k_1}&& && 0 & 0\\
\vdots&\vdots& & \hat{\Xi} & & \vdots & \vdots \\
0&-\xi_{1k_{2k-4}} && && 0 & 0\\ \hline
0&-i&0& \dots & 0 & 0 &0\\
-1&0&0 & \dots & 0 & 0 &0
\end{array}\right) = -\det(\hat{\Xi})
\]
where now $\hat{\Xi}$ is a $(2j-4)\times(2j-4)$ submatrix of $\Xi$ which {\em
does not include the first row and column}. Summing all of the $2j\times2j$
minors of $Y'_\Xi$ then yields the sum of $(2j -4)\times(2j-4)$ minors of
$\Xi$ which include the first row.

Suppose $\Xi$ is of the form:
\[
\Xi = \left(\begin{array}{c|cccc}
0       & \xi_1 & 0 & \dots & 0 \\\hline
-\xi_1  &       &   &       & \\
0       &       &   &       & \\
\vdots  &       &   & \Xi'  & \\
0       &       &   &       &

\end{array}\right)
\]
Then $\charpoly(\Xi) = \lambda\cdot\charpoly(\Xi') -
\xi_1^2\cdot\charpoly(\Xi')$ Since the nonconstant coefficients of the
characteristic polynomial are the sums of the minors, and are algebraically
independent on $\so(m)$, this tells us that the sums of minors which include
the first row and column are also algebraically independent. Putting everything
together, this tells us that for any $\tilde{F}$ in \eqref{ratbreakdown},
we can find a $Y_\Xi$ such that $Q_2(Y_\Xi) = 0$ and $\tilde{F}(Y_\Xi) \neq 0$.
Therefore, $\ell = 0$ in \eqref{ratbreakdown}, and so every polynomial
on $\gal(n)^*$ which invariant on the coadjoint orbits is generated by our
$Q_j$'s
\end{proof}
\bibliography{bib}{}
\bibliographystyle{plain}
\end{document}